\documentclass[a4paper, reqno, 11pt]{amsart}
\usepackage{amssymb}
\usepackage{hyperref}
\hypersetup{colorlinks=true,
linkcolor=blue,
filecolor=purple, 
urlcolor=cyan,
citecolor=magenta
}
\usepackage{enumitem}
\usepackage[centering]{geometry}
\usepackage{mlmodern}
\newtheorem{theorem}{Theorem}[section]
\newtheorem{lemma}[theorem]{Lemma}
\newtheorem{claim}[theorem]{Claim}

\theoremstyle{definition}
\newtheorem{definition}[theorem]{Definition}
\newtheorem{example}[theorem]{Example}
\newtheorem{corollary}[theorem]{Corollary}

\theoremstyle{remark}

\numberwithin{equation}{section}
\allowdisplaybreaks
\begin{document}

\title[On the well-posedness of linear evolution equations]{On the well-posedness of linear evolution equations under unbounded nonautonomous perturbations}

\author[X.-Q. Bui]{Xuan-Quang Bui}
\address{Faculty of Fundamental Sciences, PHENIKAA University, Nguyen Trac Street, Duong Noi Ward, Hanoi 12116, Viet Nam}
\email[corresponding author]{quang.buixuan@phenikaa-uni.edu.vn}
\thanks{The research of X.-Q.~Bui was supported by the Vietnam National Foundation for Science and Technology Development (NAFOSTED) under Grant No. 101.02-2025.56.}

\author[V.T Luong]{Vu Trong Luong}
\address{VNU University of Education, Vietnam National University, Hanoi, 144 Xuan Thuy, Cau Giay, Hanoi, Vietnam} 
\email{vutrongluong@vnu.edu.vn}

\author[N.V.~Minh]{Nguyen Van Minh}
\address{Department of Mathematics and Statistics, University of Arkansas at Little Rock, 2801 S University Ave, Little Rock, AR 72204. USA}
\email{mvnguyen1@ualr.edu}
% -----------------------------
% -----------------------------
\date{\today}
% -----------------------------
% -----------------------------
% -----------------------------
%\dedicatory{}
% =============================
% =============================
\begin{abstract} 
We study conditions for the well-posedness of nonautonomous perturbation of evolution equations of the form
\[
u'(t)=(A+B(t))u(t), \quad t \in [a,b],
\]
where $A$ generates a $\mathrm{C}_0$-semigroup $\left (T(t)\right )_{t\ge 0}$ with $\| T(t)\| \le Me^{\omega_0 t}$, $t\ge 0$, in a Banach space $\mathbb{X}$ and $B(t)$ are $t$-dependent (unbounded) linear operators in $\mathbb{X}$. 
The unbounded perturbation operators $B(t)$ are assumed to belong to a normed space (denoted by $\mathcal{GL}_A (\mathbb{X})$) of unbounded linear operators $C$ in $\mathbb{X}$ such that $D(A) \subset D(C)$ with norm 
\[
\| C\|_A:= (1/M) \sup_{\mu >\omega_0 } \| (\mu-\omega_0) CR(\mu,A)\| <\infty.
\]
We prove that the
above-mentioned evolution equation admits an evolution family if $\| B(\cdot)\|_A$ is continuous in $[a,b]$. 
The evolution family is unique if $B(\cdot)R(\mu, A)$ as a function $[a,b]\to \mathcal{L}(\mathbb{X})$ is continuously differentiable, and
\[
\limsup_{\mu \to\infty} \sup_{t\in [a,b]} \left \|  \frac{d}{dt}[B(t)R(\mu,A)]\right \| <\infty.
\]
Examples are given to illustrate the obtained results.
\end{abstract}
% =============================
% =============================
\maketitle
% =============================
% =============================
\section{Introduction}

The theory of evolution equations in Banach spaces provides a fundamental framework for studying time-dependent processes arising in partial differential equations, mathematical physics, control theory, biological processes, and many other areas of application.
A central object in this theory is the abstract linear evolution equation
\begin{equation}\label{Eq00}
u'(t)=Au(t), \quad t\ge 0,
\end{equation}
where $A$ is a (possibly unbounded) linear operator on a Banach space $\mathbb{X}$. 
It is well known that the well-posedness of evolution equation \eqref{Eq00} is closely related to the generation of a $\mathrm{C}_0$-semigroup on $\mathbb{X}$. 
The theory of $\mathrm{C}_0$-semigroups provides a powerful framework for studying such problems, see, for instance, the classical monographs \cite{engnag, hen, paz}.

A natural question concerns the stability of well-posedness under perturbations of the generator. 
More precisely, one considers the perturbed evolution equation
\begin{equation}\label{EE-Perturbed}
u'(t) = (A+C)u(t), \quad t \geq 0,
\end{equation}
where the linear operator $A$ generates a $\mathrm{C}_0$-semigroup and $C$ is a linear operator on $\mathbb{X}$ represents a certain perturbation. 
Classical results are well established in the case of bounded or relatively bounded perturbations, relying on Variation-of-Constants Formulas and semigroup techniques. 
If $C$ is bounded, the classical bounded perturbation theorem ensures that $A+C$ still generates a $\mathrm{C}_0$-semigroup, and hence the corresponding evolution equation remains well-posed, see \cite{engnag, paz}. 
However, when the perturbation terms lack uniform boundedness or exhibit explicit time dependence, the standard framework is no longer directly applicable, and additional analytical tools are required. 
We refer the reader to  \cite{bat, buihuyluomin, dunsch, engnag, hen, keiwei,meg, miy, voi}, and their references for more information.

Recently, in \cite{buihuyluomin} a new form for an unbounded perturbation problem was proposed. 
In that paper, one investigated sufficient conditions ensuring that the linear operator $A+C$ generates a strongly continuous semigroup. 
The key idea was to apply the Hille--Yosida Theorem to the perturbed linear operator, instead of using the Variation-of-Constants Formula, in order to prove the generation of a $\mathrm{C}_0$-semigroup.
In fact, one assumes that the unperturbed linear evolution equation \eqref{Eq00} is well-posed, that is, $A$ generates a $\mathrm{C}_0$-semigroup $\left (T(t)\right )_{t\ge 0}$, which satisfies the growth condition
\begin{equation}\label{TAsemigroup}
\| T(t)\|\le Me^{\omega_0 t},\quad t\ge 0.
\end{equation}
Then, the main result in \cite{buihuyluomin} shows that if the domain $D(A)\subset D(C)$ and the condition
\begin{equation}\label{1.1}
\sup_{\mu >\omega_0 } \| (\mu-\omega_0) CR(\mu,A)\| <\infty ,
\end{equation}
where \eqref{TAsemigroup} is satisfied, then the perturbed evolution equation \eqref{EE-Perturbed} is also well-posed, that is, $A+C$ generates a $\mathrm{C}_0$-semigroup. 
In other words, it is proved in \cite{buihuyluomin} that if $A$ generates a $\mathrm{C}_0$-semigroup, and if $C$ is an operator such that $D(A)\subset D(C)$, and 
\[
\limsup_{\mu\to+\infty}\| CR(\mu,A)\| <\infty,
\]
then, $A+C$ generates a $\mathrm{C}_0$-semigroup as well.

In various situations the perturbation may also depend explicitly on time. 
This leads to evolution equations of the form
\begin{equation}\label{1.2}
u'(t)=(A+B(t))u(t),\quad t\in [a,b],
\end{equation} 
where $\{B(t)\}_{t \in [a,b]}$ is a family of (possibly unbounded) operators on $\mathbb{X}$.

It is natural to ask under what conditions on the nonautonomous perturbation $B(t)$ the perturbed evolution equation \eqref{1.2} is well-posed.
The analysis of such nonautonomous perturbations is considerably more delicate, since the classical semigroup techniques are no longer directly applicable. 
Understanding conditions under which evolution equation \eqref{1.2} is well-posed therefore becomes an important and challenging problem in the theory of evolution equations.

In the mathematical literature, there are relatively few works dealing with nonautonomous perturbations of evolution equations,
except for \cite{rab} and the references therein.
To our knowledge, the conditions on the perturbation are of Miyadera-type that is different than \eqref{1.1}. 
For more information on Miyadera-type of perturbation of evolution equation the reader is referred to \cite{engnag, tan, tan2, voi} and the references therein.

In this paper we will try to give an answer to the above question by using the results obtained in \cite{buihuyluomin} to construct an evolution family $\left ( U(t,s)\right )_{a\le s\le t\le b}$ associated with Eq.~\eqref{1.2} under a condition inspired by \eqref{1.1}, namely a continuity of the family $B(t)$ with respect to $t\in [a,b]$ in the norm $\| \cdot \|_A$ defined as
\begin{equation}\label{normA}
\| B(t)\|_A:= \frac{1}{M}\sup_{\mu >\omega_0 } \| (\mu-\omega_0) B(t)R(\mu,A)\| <\infty .
\end{equation}
The continuity of $B(\cdot)$ in the norm $\| \cdot\|_A$ implies some sort of continuity of the semigroup $T_{A+B(s)}(t)$ in $s\in [a,b]$ (see Lemma \ref{lem per} below) that is important in our Euler polygon approximation process of the evolution family using piecewise linear evolution equations on $[a,b]$. Under a condition on the continuous differentiability of the family of bounded operators $B(t)R(\mu,A)$, $t \in [a,b]$, we prove that such an evolution family $\left (U(t,s)\right )_{a\le s\le t\le b}$ associated with Eq.~\eqref{1.2} is unique.

Our main idea of constructing the evolution family $\left (U(t,s)\right )_{a\le s\le t\le b}$ is to use the so-called Euler polygon method in Ordinary Differential Equations to prove the existence of a solution to the Cauchy problem. That is, by partitioning the interval $[a,b]$ with 
\[
a=t_0 <t_1<t_2<\ldots <t_N=b,
\]
we approximate Eq.~\eqref{1.2} by a piecewise linear equation 
\begin{equation}\label{1.3}
u'(t)=(A+B(t_j)) u(t),\qquad t_j \le t <t_{j+1}, \quad 0\le j\le N.
\end{equation}
On each subinterval $\left [t_j,t_{j+1}\right )$, it follows from our recent result in \cite{buihuyluomin} that the operator $A+B(t_j)$ generates a $\mathrm{C}_0$-semigroup. Therefore, Eq.~\eqref{1.3} can be integrated on $\left [t_j,t_{j+1}\right )$ by means of this semigroup.
We prove that under continuity condition \ref{ItemA1} on the family $B(t)$ in the sense of \eqref{normA} depends continuously on $t\in [a,b]$, when refining the partition the obtained evolution family from Eq.~\eqref{1.3} will be convergent to an evolution family associated with Eq.~\eqref{1.2}. 
Further, under \ref{ItemA2} on the continuous differentiability of the family of bounded operators $B(t)R(\mu,A), \ t \in [a,b]$ such an evolution family associated with Eq.~\eqref{1.2} is unique.
This result extends the main result obtained in \cite{buihuyluomin}.

Although our approach is also based on the Euler polygon construction used in \cite[Theorem~3.1, Chapter~5]{paz}, the assumptions in the present paper are essentially different. 
In the classical framework, one typically requires stability-type conditions for the operator family together with certain uniform continuity assumptions. 
In contrast, our setting only assumes continuity with respect to the metric of $\mathcal{GL}_A(\mathbb{X})$. 
The analysis is then mainly based on estimates involving the Yosida distance and the distance induced by $\mathcal{GL}_A(\mathbb{X})$. 
This shows that these distances provide a natural and efficient framework for treating perturbation problems in the theory of linear evolution equations. 

The main result of this paper is stated in Theorem~\ref{the main}, where we establish the well-posedness of evolution equations under unbounded nonautonomous perturbations. 
As consequences of this result, we further obtain well-posedness on the whole real line (Theorem~\ref{the R}) and a roughness result for exponential dichotomies (Theorem~\ref{RoughnessEDtheR}).
For illustration, we give some examples at the end of the paper.

\section{Preliminaries}
\subsection{Notations}
Throughout this paper, $\mathbb{X}$ denotes a Banach space. 
As usual, $\mathbb{N}$ and $\mathbb{R}$ stand for the sets of natural numbers and real numbers, respectively.
The space of all bounded  linear operators from a Banach space $\mathbb{X}$ to itself is denoted by $\mathcal{L} (\mathbb{X})$.
The resolvent set of a linear operator $T$ in a Banach space will be denoted by $\rho(T)$.
For $\lambda \in \rho (T)$, we denote the inverse $(\lambda I - T)^{-1}$ by $R(\lambda, T)$ and call it the \textit{resolvent} of the linear operator $T$ at $\lambda$. 
Moreover, $\mathbb{X}^*$ denotes the dual space of $\mathbb{X}$, and we write $\langle x^*, x \rangle$ for the duality pairing between $x \in \mathbb{X}$ and $x^* \in \mathbb{X}^*$, that is,
$\langle x^*, x \rangle = x^*(x)$. 
We use $\varepsilon (h)$ to denote a function that approaches $0$ as $h \to 0$.
Let $A_\lambda$ denote the Yosida approximation of the operator $A$, and let $d_Y(A,B)$ denote the Yosida distance between two possibly unbounded operators, see \cite{buimin} for the definition of this notion.

\subsection{Some distances}
In this subsection, we present the concept of Yosida distance and results about linear perturbation of evolution equations (see \cite{buimin} and references therein for more information on the matter).

Recall that given an operator $A$ (that may be unbounded) in a Banach space $\mathbb{X}$ such that $\rho (A)$ and $ \rho(B)$ contain the ray $[\omega,\infty)$ for some $\omega \in \mathbb{R}$, the \textit{Yosida approximation} $A_\lambda$ of $A$ is defined to be the operator 
\[
\lambda^2 R(\lambda ,A)-I, \quad
\text{ where } \lambda > \omega.
\]

\begin{definition}[{see \cite{buimin, buimin2}}]
The \textit{Yosida distance} between two linear operators $A$ and $B$ satisfying $\rho(A) \supset [\omega, \infty)$ and $ \rho(B) \supset [\omega, \infty)$, where $\omega$ is a given number, is defined to be
\[
d_Y(A,B):=  \limsup_{\lambda\to +\infty}\| A_\lambda -B_\lambda \| = \limsup_{\lambda\to +\infty} \lambda^2\| R(\lambda,A)-R(\lambda, B)\|.
\]
\end{definition}
Let $A$ be the generator of a $\mathrm{C}_0$-semigroup $\left (T(t)\right )_{t \geq 0}$ satisfying 
\[
\| T_A(t)\| \le Me^{\omega_0 t},
\quad
t\ge 0.
\]
\begin{definition}[{see \cite{buihuyluomin}}]
Let us define $\mathcal{GL}_A(\mathbb{X})$ to be the space of all linear operator $C$ on $\mathbb{X}$ with $D(C)= D(A)$ and 
\[
\| C\|_A:=\frac{1}{M} \sup_{\mu >\omega_0 } \| (\mu-\omega_0) CR(\mu,A)\| <\infty .
\]
\end{definition}
The space $\mathcal{GL}_A(\mathbb{X})$ is a normed space (with norm $\|\cdot\|_A$) that includes $\mathcal{L}(\mathbb{X})$ and some classes of unbounded operators. 
Examples of unbounded operators will be provided in the Section~\ref{Sect-Exa} below. 
If $B,\,C\in \mathcal{GL}_A(\mathbb{X})$ another distance can be naturally defined 
\[
d_M(B,C):=\| B-C\|_A.
\]
Below is a relation between these distances.
\begin{lemma}
The following assertions are valid:
\begin{enumerate}
\item
Let $B$ and $C$ be the generators of $\mathrm{C}_0$-semigroups. Assume further that $D(B)= D(C)$. Then, $B=C$, provided that 
$d_{Y}(B, C)=0$.
\item\label{MQ-Lem12}
Let $B,\,C\in \mathcal{GL}_A(\mathbb{X})$. Then,
\[
d_Y(B,C) \le \| B-C\|_A .
\]
\item\label{MQ-Lem13}
Let $B,\, C\in \mathcal{L}(\mathbb{X})$. Then 
\[
d_{Y}(B, C)= \| B-C\|_A = \| B-C\| .
\]
\end{enumerate}
\end{lemma}
\begin{proof}
\begin{enumerate}
\item This is well known, see \cite{paz}.

\item This estimate was proved in \cite{buihuyluomin}.

\item When $B,\,C \in \mathcal{L}(\mathbb{X})$, it is proved in \cite{buimin2} that $d_Y(B,C)=\| B-C\|$. Combining this and Part~\eqref{MQ-Lem12} proves Part~\eqref{MQ-Lem13}.
\end{enumerate}
The proof is complete.
\end{proof}

\subsection{Autonomous perturbation} 
Let $\mathbb{X}$ be a Banach space. Below we fix an operator $A$ as the generator of a $\mathrm{C}_0$-semigroup $\left (T_A(t)\right )_{t\ge 0}$ in $\mathbb{X}$ with $\left \| T_A(t)\right \| \le Me^{\omega_0 t}$, $t\ge 0$.

The following generation theorem for pertubed semigroups was proved in \cite{buihuyluomin}  that is the basis for our paper:
\begin{theorem}[{see \cite{buihuyluomin}}]
\label{the gen}
If $C\in \mathcal{GL}_A(\mathbb{X})$, then the linear operator $A+C$ generates a $\mathrm{C}_0$-semigroup $\left (T_{A+C}(t)\right )_{t\ge 0}$ with 
\[
\left \| T_{A+C}(t)\right \| \le Me^{\left (\omega_0 +M^2 \| C\|_A\right )t}, \quad t\ge 0.
\]
\end{theorem}

\section{Well-posedness of evolution equations under unbounded nonautonomous perturbations}
This section is devoted to the presentation of the main results of the paper. After establishing several auxiliary lemmas, we state and prove the well-posedness of evolution equations under unbounded nonautonomous perturbations.

\subsection{Auxiliary lemmas}
We consider the perturbation of the well-posed evolution equation
\[
u'(t)=Au(t),
\]
by a nonautonomous perturbation $B(t)$, that is,
\begin{equation}\label{2}
u'(t)=(A+B(t))u(t), \quad t\in [a,b],
\end{equation}
where the linear operator $A$ is assumed to generate a $\mathrm{C}_0$-semigroup $\left (T(t)\right )_{t\ge 0}$ with
\begin{equation}\label{3}
\| T(t)\| \le Me^{\omega_0 t}, \quad t\ge 0.
\end{equation}

We will use the following assumptions as conditions for the existence and uniqueness of solutions to the Cauchy problem associated with \eqref{2}.

\begin{definition}
Given a $\mathrm{C}_0$-semigroup $\left (T(t)\right )_{t\ge 0}$ satisfying \eqref{3}, suppose that for each $t\in [a,b]$ an operator $B(t): D(B(t))= D(A) \to \mathbb{X}$ is defined.
Then, Eq.~\eqref{2} is said to satisfy 
\begin{enumerate}[label=\textnormal{\bf Assumption~(A\arabic*)},
ref=\textnormal{Assumption~(A\arabic*)},
leftmargin=*]
\item\label{ItemA1}
if the mapping $B(\cdot)\colon [a,b]\to  \mathcal {GL}_A(\mathbb{X})$ is continuous.
\item\label{ItemA2}
if the mapping $B(\cdot)R(\mu , A) \colon [a,b]\to \mathcal{L}(\mathbb{X})$ is continuously differentiable, and
\begin{equation*}
\limsup_{\mu \to\infty} \sup_{t\in [a,b]}\left \|  \frac{d}{dt}[B(t)R(\mu,A)]\right \| <\infty .
\end{equation*}
\end{enumerate}
\end{definition}

\begin{lemma}\label{lem 3.2}Let $B(\cdot)$ satisfy
\ref{ItemA1} and \ref{ItemA2}.
Then,
\begin{equation*}
\limsup_{\mu\to\infty}\sup_{t\in [a,b]}  \left\| \frac{d}{dt} R(\mu, A+B(t)) \right\|  =0.
\end{equation*}
\end{lemma}
\begin{proof}
We can verify the identity
\begin{equation}\label{a1}
R(\mu,A+B(t))=R(\mu,A)[I-B(t)R(\mu,A)] ^{-1}.
\end{equation}
and
\begin{equation*}
\frac{d}{dt}[I-B(t)R(\mu,A)] ^{-1}=	R(\mu, B(t)R(\mu,A))\cdot \frac{d}{dt} [B(t) R(\mu,A)] \cdot R(\mu, B(t)R(\mu,A)).
\end{equation*}
Therefore,
\begin{align*}
& \frac{d}{dt}	R(\mu,A+B(t))
\\ 
&=
R(\mu,A) \frac{d}{dt}[I-B(t)R(\mu,A)] ^{-1}\\
&= R(\mu,A) 
\left [
R(\mu, B(t)R(\mu,A)) \cdot \frac{d}{dt}[B(t)R(\mu,A)] \cdot R(\mu, B(t)R(\mu,A))
\right ].
\end{align*}
Consequently, if we denote
\[
K :=
\sup_{t\in [a,b]} \| B(t)\|_A= \sup_{t\in [a,b]}  \frac{1}{M}\sup_{\mu>\omega_0} (\mu-\omega_0) \| B(t)R(\mu,A)\| ,
\]
then, by \ref{ItemA1}, $K<\infty$. Next, by \ref{ItemA2}, we have
\[
\limsup_{\mu \to\infty} \sup_{t\in [a,b]} \left \|  \frac{d}{dt}[B(t)R(\mu,A)]\right \| <\infty .
\]
Therefore,
\begin{align*}
& \left\| 	\frac{d}{dt}	R(\mu,A+B(t)) \right\| 
\\
&\le 
\| R(\mu,A)\| \cdot \| R(\mu, B(t)R(\mu,A))\|^2 \cdot \left \| \frac{d}{dt}[B(t)R(\mu,A)]\right \| \\
&\le \frac{M}{(\mu-\omega_0)^2}  \cdot \| R(\mu, B(t)R(\mu,A))\|^2 \cdot \sup_{t\in [a,b]}\left \|  \frac{d}{dt}[B(t)R(\mu,A)]\right \|\\
&\le  \frac{M}{(\mu-\omega_0)^2} \cdot \frac{1}{(1-K/(\mu-\omega_0))^2}  \cdot \sup_{t\in [a,b]}\left \| \frac{d}{dt}[B(t)R(\mu,A)] \right \|.
\end{align*}
Clearly,
\[
\lim_{\mu\to\infty}  \sup_{t\in [a,b]} \left\| 	\frac{d}{dt}	R(\mu,A+B(t)) \right\|  =0.
\]
The proof is complete.
\end{proof}

Below we will need some preparatory results that are the key to estimating approximation processes later.
\begin{lemma}\label{lem key}
Let $a_1,a_2,\ldots, a_N$ and $b_1,b_2,\ldots, b_N$ be two sequences of bounded linear operators in a Banach space such that for each $ 1\le j\le N$,
\[
0\le 	\| a_j \| \le K, \quad
0\le \|	b_j \| \le K, \quad
\| a_j-b_j\|  \le\delta,
\]
where $K \geq 1$ and $\delta > 0$ are some positive constants depending only on the sequences.
Then,
\[
\left\| 	\prod_{j=1}^N a_j -	\prod_{j=1}^N b_j\right\| \le N\delta K^{N-1}.
\]
\end{lemma}
\begin{proof}
We will prove this claim by induction. When $N=1$ the claim is obvious. Assume that the claim holds for $N=m$. We will prove it is true for $N=m+1$.
By the induction assumption, we have
\begin{align*}
\left\| 	\prod_{j=1}^{m+1} a_j -	\prod_{j=1}^{m+1} b_j\right\|  &=	\left\| 	\prod_{j=1}^{m+1} a_j -	 a_1 \prod_{j=2}^{m} b_j+      a_1\prod_{j=2}^{m} b_j  -  \prod_{j=1}^{m+1} b_j\right\|  \\
&\le \|  a_1 \|  \left( \left\| 	\prod_{j=2}^m a_j -	  \prod_{j=2}^m b_j \right\| \right) +     \| a_1-b_1\|  \prod_{j=2}^m \| b_j \| \\
&\le K\times m\delta K^{m-1} + \delta K^m\\
&= (m+1)\delta K^m .
\end{align*}
This completes the proof of the claim.
\end{proof}

\begin{lemma}\label{lem per}
Let $G$ and $H$ be the generators of $\mathrm{C}_0$-semigroups such that
\begin{enumerate}
\item the Yosida distance $d_Y(G,H)$ is finite;
\item there exist positive constants $M$ and $\omega$ such that for all $t\ge 0$, 
\[
\| T_G(t)\| \le Me^{\omega t}, 
\quad
\| T_H(t)\| \le Me^{\omega t}.
\]
\end{enumerate} 
Then, for each $0\le s\le t$, $x\in \mathbb{X}$, the following estimate holds
\begin{equation}\label{ab}
\left \| T_G(t-s)x-T_H(t-s)x \right \| \le
|t-s|M^2e^{4\omega |t-s|}d_Y(G,H) \| x\|.
\end{equation}
\end{lemma}
\begin{proof}
The proof is based on deriving estimates for the Yosida approximations and then passing to the limit.
To this end, we first consider bounded operators.

Let $C$ and $D$ be two bounded linear operators in a Banach space $\mathbb{X}$. We will estimate the growth of
\[
e^{(t-s)C}-e^{(t-s)D}.
\]
Applying the Variation-of-Constants Formula to the evolution equation
\[
x'(t)=Cx(t)+(D-C)x(t)
\]
and by setting 
\[
x(t)=e^{(t-s)D}x
\]
we have
\[
x(t)=e^{(t-s)C}x+\int^t_s e^{(t-\xi)C}(D-C)x(\xi)d\xi.
\]
For each $t\ge s$, we have
\begin{align*}
\left \|e^{(t-s)C}x-e^{(t-s)D}x\right \|
&\le
\int^t_s \left \| e^{(t-\xi)C}(D-C)e^{\xi D}x\right \|d\xi
\\
&\le
(t-s)\| C-D\|  
\left \| e^{tC}\right \|
\left \| e^{tD}\right \|
\| x\|.
\end{align*}

Now if we let $G_\lambda$ and $H_\lambda$ be the Yosida approximations of $G$ and $H$, then, (see, e.g., Pazy~\cite[(5.25)]{paz}) as
\[
\left \| e^{tG_\lambda } \right \|
\le  Me^{2\omega t},
\quad
\left \| e^{tH_\lambda }\right \|
\le  Me^{2\omega t}
\]
we have
\[
\left \| e^{(t-s)G_\lambda}x-e^{(t-s)H_\lambda }x\right \|
\le 	(t-s)\| G_\lambda-H_\lambda \|  
\left \| e^{(t-s)G_\lambda}\right \|
\left \| e^{(t-s)H_\lambda}\right \|
\| x\|.
\]
Therefore,
\begin{align*}
\limsup_{\lambda \to +\infty} \left \| e^{(t-s)G_\lambda}x-e^{(t-s)H_\lambda }x\right \| & \le
(t-s)M^2e^{4\omega (t-s)} \limsup_{\lambda \to\infty } 
\left \| G_\lambda  - H_\lambda \right \|   \| x\| 
\\
&=
(t-s)M^2e^{4\omega (t-s)} d_{Y}(G,H)\| x\|.
\end{align*}
Consequently, as 
\begin{align*}
\lim_{\lambda \to +\infty}e^{(t-s)G_\lambda}x &=T_G(t-s)x,\\
\lim_{\lambda \to +\infty}e^{(t-s)H_\lambda}x &=T_H(t-s)x
\end{align*}
the estimate \eqref{ab} is valid.
\end{proof}

\subsection{Main results}
We are now in a position to state and establish the well-posedness of the evolution equation in the presence of unbounded perturbations.

\begin{definition}
Equation~\eqref{2} is said to \textit{admit an evolution family $\left (U(t,s)\right )_{a\le s\le t\le b}$} if $\left (U(t,s)\right )_{a\le s\le t\le b}$ is a family of bounded linear operators in a Banach space $\mathbb{X}$ that satisfies the following conditions:
\begin{enumerate}
\item It is strongly continuous, that is, the map $U(t,s)x$ is continuous in $(t,s)$ with $a\le s\le t\le b$ for each fixed $x\in \mathbb{X}$, and
\begin{align*}
& U(t,t) =I,
\\
& U(t,s)U(s,r) =U(t,r) , \qquad r\le s \le t, \quad r,\, s,\, t \in [a,b],
\\
& \| U(t,s) \| \le Me^{\omega (t-s)}, \quad a\le s\le t \le b .
\end{align*}
\item For each $x\in D(A)$, $U(t,s)x$ is differentiable with respect to $t\ge s$, and
\begin{align}
\frac{\partial^+ U(t,s)}{\partial t} x \Big|_{t=s} &=A(s)x,\label{3.7} \\
\frac{\partial U(t,s)}{\partial s} x  \Big|_{t=s} &=-U(t,s)A(s)x. \label{3.8} 
\end{align} 
\end{enumerate}
\end{definition}

Without loss of generality, by rescaling and re-norming procedures, we will assume below that 
\[
[a,b]=[0,1]
\]
and the fixed $\mathrm{C}_0$-semigroup $\left (T_A(t)\right )_{t \geq 0}$ satisfies 
\[
\left \| T_A(t)\right \| \le 1.
\]
Since $B(t)\in \mathcal{GL}_A(\mathbb{X})$ and is continuous in $t\in [0,1]$
\[
\sup_{t\in [0,1]} \| B(t)\|_A =\omega_1 <\infty .
\]
By Theorem \ref{the gen}, for each $r\in [0,1]$ we have
\[
\left \| T_{A(\xi )}(t)\right \| \le e^{\omega_1t}, \ t\ge 0 ,
\]
where 
\[
A(\xi ):= A+B(\xi )
\quad \text{ for each }
\xi \in [0,1].
\]

We will prove the following theorem as the main result of the paper.
\begin{theorem}\label{the main}
Let $A$ generate a $\mathrm{C}_0$-semigroup satisfying \eqref{3}. Assume that $B(\cdot )$ is a family of operators in $\mathbb{X}$. 
Then, the following assertions are true:
\begin{enumerate}
\item {\bf Existence}. Under \ref{ItemA1}, the perturbed equation \eqref{2} admits an evolution family $\left (U(t,s)\right )_{t\ge s}$.
\item {\bf Uniqueness}. Under \ref{ItemA1} and \ref{ItemA2}, the above-mentioned evolution family $\left (U(t,s)\right )_{t\ge s}$ is unique.		
\end{enumerate}
\end{theorem}
\begin{proof}
We proceed in two steps, corresponding to the two assertions in the statement of the theorem.

\paragraph{\textbf{Existence of evolution family.}}
We will use the idea of Euler polygon curves in the proof of the Existence and Uniqueness Theorem of ODE to construct sequences of approximated solutions to the Cauchy problem that leads to an evolution family $\left (U(t,s)\right )_{0\le s\le t\le 1}$ for Eq.~\eqref{2}. That is, if we partition the interval $[a,b]$ by a finite sequence of numbers
\[
t_0=0 <t_1<t_2 < \cdots <t_N=1,
\]
then, on each subinterval $[t_j,t_{j+1})$ we approximate $A(t):= A+B(t)$ in Eq.~\eqref{2} by the piecewise constant family $A(t_j)$ for all $t\in [t_j,t_{j+1})$. Then, if refine the partition, the obtained evolution family will be convergent.

For each $n\in \mathbb{N}$, we divide $[0,1]$ into 
\[
N:= 2^n
\]
subintervals of equal length. 
We construct an evolution family $\left (U_n(t,s)\right )_{0 \leq s \leq t \leq 1}$ based on the Euler polygon procedure by the partition
\begin{equation}\label{QM-partt2n}
t_{j}^{(n)} := \frac{j}{2^n}, \quad j=0,1,2,\cdots , N=2^n.
\end{equation}
An explicit formula for the evolution family associated with the above mentioned Euler polygon is given in the proof of \cite[Theorem 3.1, Chapter 5]{paz}. Namely, considering the partition
\eqref{QM-partt2n}
we define a two-parameter family of operators $U_n(t,s)$, $0 \le s \le t \le 1$, by
\begin{equation}
\label{Pazy-3.5}
U_n(t,s) =
\begin{cases}
T_{A\left (t_{j}^{(n)}\right )}(t-s), 
&
\text{ for } 
t_{j}^{(n)} 
\le s \le t \le 
t_{j+1}^{(n)},
\\
T_{A\left (t_{k}^{(n)}\right )} \left(t - 
t_{k}^{(n)}
\right)
&
\hspace*{-4mm}
\displaystyle
\left [
\prod_{j=l+1}^{k-1} T_{A(t_{j}^{(n)})}
\left (t_1^{(n)}\right )\right ]
T_{A\left (t_{l}^{(n)}\right )}  \left(
t_{l+1}^{(n)}
- s\right), 
\\
& \text{ for } j<k,\, t_{k}^{(n)} \le t \le t_{k+1}^{(n)},
\\
& \text{ and } t_{l}^{(n)} \le s \le t_{l+1}^{(n)}.
\end{cases}
\end{equation}
Using basic properties of $\mathrm{C}_0$-semigroups and Lemma~\ref{lem key} we can easily verify  that $U_n(t,s)$ is an evolution family (that is  associated with evolution equation the piecewise right-hand side
\[
u'(t) = A\left (t_{j}^{(n)}\right ) u(t),\qquad t\in \left [t_{j}^{(n)}, t_{j+1}^{(n)}
\right ), 
\quad 0\le j \le 2^n-1,
\]
that is, it satisfies the following properties:
\begin{align*}
& U_n(s,s) = I,
\\
& U_n(t,s) = U_n(t,r)U_n(r,s), \quad 0 \le s \le r \le t \le 1,\\
& (t,s) \mapsto U_n(t,s) \text{ is strongly continuous on } 0 \le s \le t \le 1,
\\
& \|U_n(t,s)\| \le e^{\omega_1 (t-s)}, \quad 0 \le s \le t \le 1.
\end{align*}
For each $m \ge n$, by our construction, each interval $\left [t^{(n)}_k,t^{(n)}_{k+1}\right ]=\left [\frac{k}{2^n}, \frac{k+1}{2^n}\right ]$ is divided into $2^{m-n}$ subintervals of equal length. Note that since we consider $n$ and $m$  to be sufficiently large, we may assume that
\[
t-s > \frac{1}{2^m}.
\]
The formulas for $U_m(t,s)$ is obtained by replacing each factor in \eqref{Pazy-3.5} with an appropriate expressions. For example,
\[
T_{A\left (t_{j}^{(n)}\right )}
\left(t_{1}^{(n)}\right)
\text{ replaced with }
\prod_{t_{j}^{(n)} \le t_{i}^{(m)} \le t_{j+1}^{(n)}} 
T_{A\left (t_{i}^{(m)}\right )}
\left(t_{1}^{(m)}\right).
\]
By Lemma \ref{lem per} (with $M=1$ and $\omega =\omega_1$), since
\[
\left\|  T_{A\left (t_{i}^{(m)}\right )}
\left(t_{1}^{(m)}\right) - T_{A\left (t_{i}^{(n)}\right )}
\left(t_{1}^{(m)}\right) 
\right\| 
\le 
\frac{1}{2^{m}} \Omega_m e^{ 4\omega_1 2^{-m}}.
\]
Let us denote for $t_{j}^{(n)} \le t_{i}^{(m)} \le t_{j+1}^{(n)}$
\[
a_i := T_{A\left (t_{i}^{(m)}\right )}
\left(t_{1}^{(m)}\right),
\quad
b_i :=T_{A\left (t_{j}^{(n)}\right )}
\left(t_{1}^{(m)}\right)
\]
and 
\[
N:=2^{m-n}, 
\quad 
\delta :=\frac{1}{2^m} \Omega_m e^{ 4\omega_1 2^{-m}},
\quad 
K:=e^{2^{-m}\omega_1}.
\]
Then, 
\begin{align*}
&
\prod_{i=1}^Na_i = \prod_{t_{j}^{(n)} \le t_{i}^{(m)} \le t_{j+1}^{(n)}} 
T_{A\left (t_{i}^{(m)}\right )}
\left(t_{1}^{(m)}\right),
\\
&
\prod_{i=1}^Nb_i = \prod_{t_{j}^{(n)} \le t_{i}^{(m)} \le t_{j+1}^{(n)}}  T_{A\left (t_{j}^{(n)}\right )}
\left(t_{1}^{(m)}\right).
\end{align*}
By Lemma \ref{lem key} 
we have
\begin{align*}
&
\left\| 
T_{A\left (t_{j}^{(n)}\right )}
\left(t_{1}^{(n)}\right) -	\prod_{t_{j}^{(n)} \le t_{i}^{(m)} \le t_{j+1}^{(n)}} 
T_{A\left (t_{i}^{(m)}\right )}
\left(t_{1}^{(m)}\right)
\right\| 
\\	
&\le 
2^{m-n}
\left[ \frac{1}{2^m} \Omega_n e^{2^{{-m}} \times 4\omega_1} \right] e^{2^{m-n}\times 2^{-m}4\omega_1 } \\
&\le  
\frac{1}{2^n}  \Omega_n e^{8\omega_1(2^{-n})} .
\end{align*}
Applying Lemma \ref{lem key} again to the following sequences with 
\[
\delta = \frac{1}{2^n}  \Omega_n e^{8\omega_1(2^{-n})} ,
\quad 
K=e^{ 2^{-n}4\omega_1 },
\quad 
N=\left [(t-s)2^n\right ],
\]
where $[\alpha]$ denotes the integral part of $\alpha$, corresponding to the partition of $[s, t]$ into intervals of length $\frac{1}{2^n}$,
gives
\begin{equation}\label{Cauchy}
\left \| U_n(t,s) -	U_m(t,s) \right \| 
\le 
(t-s) e^{4\omega_1} \Omega_n.
\end{equation}
Since the function $B(\cdot ):[0,1]\to \mathcal{GL}_{A}(\mathbb{X})$ is uniformly continuous, the oscillation \[
\Omega_n\to 0
\quad \text{ as }
n\to \infty,
\]
so the sequence $\left (U_n(t,s)\right )_{0 \leq s \leq t \leq 1}$ is a Cauchy sequence that converges to $\left (U(t,s)\right )_{0 \leq s \leq t \leq 1}$. 
We now prove that $\left (U(t,s)\right )_{0 \leq s \leq t \leq 1}$ satisfies all properties we need.

First, $\left (U(t,s)\right )_{0 \leq s \leq t \leq 1}$ is an evolution family as this property follows from the similar property of $\left (U_n(t,s)\right )_{0 \leq s \leq t \leq 1}$.

Next, we prove \eqref{3.7}, that is, for each $x\in D(A)$, $U(t,s)x \in D(A)$ and
\begin{equation}\label{3.24}
\left .
\frac{\partial^+ U(t,s)}{\partial t} x \right |_{t=s} =A(s)x.
\end{equation}
To this end, we prove the following claim
\begin{claim}
For each $r\in [0,1]$, $a\le r\le t \le b$, we have
\begin{equation}\label{UTAr}
\sup_{0\le r \le t\le 1} \frac{1}{t-r} 	\left \| U(t,r)x-T_{A(r)}(t-r)x \right \|  =  \varepsilon (|t-r|) ,
\end{equation}
where $\varepsilon (\delta)$ denote a function approaching $0$ as $\delta \to 0$.
\end{claim}
\begin{proof}
If we repeat the above mentioned argument when proving \eqref{Cauchy},
then,
\[
\left \| 	U_n(t,r) -	U_m(t,r) \right \| 
\le (t-r) e^{4\omega_1} \Omega_n .
\]
In the construction of $U_n(t,r)$ we assume $n$ is a large number when $t,\, r$ are fixed.
Now letting $m\to \infty$ gives
\begin{equation}\label{3.27}
\left \| 	U_n(t,r) -	U(t,r) \right \| 
\le |t-r| e^{4\omega_1} \Omega_n .
\end{equation}	
By Lemmas \ref{lem key} and \ref{lem per} we have
\begin{align}\label{3.28}
\left \| U_n(t,r)x-T_{A(r)}(t-r)x \right \|  &\le |t-r| e^{4\omega_0|t-r|} \Omega_{[r,t]} ,
\end{align}
where 
\[
\Omega_{[r,t]}=\sup_{[\xi \in r,t]} \| B(r)-B(\xi)\|_A.
\]
Combining \eqref{3.27} and \eqref{3.28} gives
\begin{align*}
& 
\sup_{0\le r \le t\le 1} \frac{1}{t-r} 	\left \| U(t,r)x-T_{A(r)}(t-r)x \right \|  
\\
&\le 
\sup_{0\le r \le t\le 1} \frac{1}{t-r}  	\left \| U_n(t,r)x-T_{A(r)}(t-r)x \right \| 
+ 
\sup_{0\le r \le t\le 1} \frac{1}{t-r}  \left \| 	U_n(t,r)x -	U(t,r)x \right \| 
\\
&\le  e^{4\omega_0} \left (\Omega_{[r,t]}+\Omega_n\right ) .
\end{align*}
Note that in these formulas 
\[
|t-r| \ge \frac{1}{2^n},
\]
so if $|t-r|\to 0$ as $n\to \infty$,
and by the uniform continuity of $B(\cdot)$ in $\mathcal{GL}_A(\mathbb{X})$, the quantities 
\[
\Omega_n,\, \Omega_{[r,t]}  \to 0.
\]
This proves the claim.
\end{proof}

Now we are ready to prove \eqref{3.24}.
We have, for $h>0$,
\begin{align*}
& \left\| \frac{1}{h}\left( U(s+h,s)x-x \right) -A(s)x\right\| 
\\
&= 	\left\| \frac{1}{h}\left( U(s+h,s)x-T_{A(s)}(h)x \right) -  \frac{1}{h}\left( T_{A(s)}(h)x-x)-A(s)x \right) \right\| \\
&\le 	\left\| \frac{1}{h}\left( U(t+h,t)x-T_{A(s)}(h)x \right) \right\|  +
\left\|   \frac{1}{h}\left( T_{A(s)}(h)x-x)-A(s)x \right) \right\|.
\end{align*}
Since $x\in D(A)$ and by \eqref{UTAr}
\begin{align*}
\left\| \frac{1}{h}\left( U(s+h,s)x-T_{A(s)}(h)x \right) \right\|  & = \varepsilon_1 (h),
\\
\left\|   \frac{1}{h}\left( T_{A(s)}(h)x-x)-A(t)x \right) \right\|  &= \varepsilon_2 (h),
\end{align*}
where $\varepsilon_1,\,\varepsilon_2\to 0$ as $h\to 0$.
This yields that
\[
\left .
\frac{\partial^+ U(t,s)}{\partial t} x\right |_{t=s} =A(s)x.
\]
Similarly, we can prove \eqref{3.8}.

\paragraph{\textbf{Uniqueness.}}
Suppose that $(V(t,s))_{t\ge s}$ is another evolution family that satisfies all properties in the theorem. For each pair $0\le s\le t \le 1$, for $s\le \xi \le t$ and $x\in D(A)$ set 
\[
u(\xi )  := U(t,\xi)x- V(t,\xi)x.
\]
We are going to show that 
\[
u(\xi )=0,
\quad \text{ for all }
\xi \in [0,1],
\]
or equivalently, for all $\xi\in[0,1]$, 
\[
R(\lambda, A(\xi))u(\xi)=0.
\]
By our supposition, since $u'(t)=A(t)u(t)$,
\begin{align*}
\frac{d}{dt} \left( R(\lambda ,A(t))u(t) \right)  &= \frac{d}{dt} \left( R(\lambda ,A(t)) \right)  u(t) + R(\lambda ,A(t))u'(t)\\
&=  \frac{d}{dt} \left( R(\lambda ,A(t)) \right)  u(t) + R(\lambda ,A(t))A(t)u(t) \\
&=  \frac{d}{dt} \left( R(\lambda ,A(t)) \right)  u(t)  +\left[\lambda R(\lambda,A(t))-I\right]u(t) \\
&= \left[ \lambda R(\lambda,A(t)) +  \frac{d}{dt} \left( R(\lambda ,A(t)) \right)  -I\right] u(t).
\end{align*}
Set
\begin{align*}
v(t)&= R(\lambda ,A(t))u(t), \\
f(t)&=\left[ \frac{d}{dt} \left( R(\lambda ,A(t)) \right)  -I\right] u(t) .
\end{align*}
Then, $v(t)$ is a classical solution of the equation
\[
v'(t) =\lambda  v(t)+f(t).
\]
Therefore, by the Variation-of-Constants Formula,
\[
R(\lambda ,A(t))u(t)=\int_0^t e^{\lambda (t-s)} \left[ \frac{d}{ds} \left( R(\lambda ,A(s)) \right)  -I\right] u(s)ds .
\]
By \eqref{a1}  since for all $t\in [0,1]$
\[
\| R(\lambda ,A(t)) \| \le \frac{C}{\lambda -\omega_0},
\]
where $C$ is some constant, we have
\begin{equation}\label{3.49}
\lim_{\lambda \to \infty} \int_0^t e^{\lambda (t-s)} \left[ \frac{d}{ds} \left( R(\lambda ,A(s)) \right)  -I\right] u(s)ds =0.
\end{equation}

Next, we  will  prove $u(s)=0$ for all $s\in [0,1]$. In fact,  for any fixed $x^*\in \mathbb{X}^*$ set
\[
w_\lambda (s):=
\left \langle
x^*, \left[ \frac{d}{ds} \left( R(\lambda ,A(s)) \right)  -I\right] u(s)\right \rangle.
\]
Then, by \eqref{3.49} and Lemma \ref{lem 3.2}
\begin{align}
\sup_{n\in\mathbb{N}} \left| 	\int_0^t e^{n (t-s)} w_n(s)ds \right| &=: M_1 <\infty,
\label{3.50}
\\
\sup_{n\in\mathbb{N}} 	\sup_{s\in [0,1]} | w_n(s) | &=: M_2 < \infty.
\notag
\end{align}
Below we will use the idea in the proof of \cite[Lemma~1.1, p. 100]{paz} to exploit this estimate \eqref{3.50} to prove $u(t)=0$ for any $t\in [0,1]$.
Consider the series
\begin{equation}\label{ser}
\sum_{k=1}^\infty \frac{(-1)^{k-1}}{k!} e^{kn\tau} =1-e^{-e^{n\tau}}.
\end{equation}
Notice that this series is absolutely convergent on each bounded interval of $\tau$, so for each $0\le t <1$ since
\begin{align*}
\left| \int^1_0 \sum_{k=1}^\infty \frac{(-1)^{k-1}}{k!} e^{kn(t-1+s)} w_n(s)ds \right| 
& \le \sum_{k=1}^\infty \frac{1}{k!} e^{kn(t-1)} \left| \int^1_0 e^{kns}w_n(s) ds \right| \\
& \le M_1 \left (e^{e^{n(t-1)}}-1\right ),
\end{align*} 
and
\[
\lim_{n\to\infty} M_1 \left (e^{e^{n(t-1)}}-1\right )=0,
\]
we have
\begin{align*}
\lim_{n\to \infty}
\int^1_0 \left (1-e^{-e^{n(t-1+s)}}\right ) w_n(s)ds 
&=
\lim_{n\to\infty} 	\int^1_0 \sum_{k=1}^\infty \frac{(-1)^{k-1}}{k!} e^{kn(t-1+s)} w_n(s)ds 
\\
&=
0.
\end{align*}
Notice that for each $0\le t< 1$ by \eqref{ser} the sequence of functions
\[
f_n(s) := \left (1-e^{-e^{n(t-1+s)}}\right ) w_n(s)
\]
has a limit $f(s)$ as $n\to \infty$, where, by Lemma~\ref{lem 3.2},
\begin{align*}
f(s)
& :=
\lim_{n\to\infty} f_n(s)
\\
&= 
\lim_{n\to \infty} \left (1-e^{-e^{n(t-1+s)}}\right ) w_n(s)\\
&=\lim_{n\to \infty} \left (1-e^{-e^{n(t-1+s)}}\right )\left \langle x^*,  \left[ \frac{d}{ds} \left( R(n ,A(s)) \right)  -I\right] u(s)\right  \rangle\\
&= \begin{cases}
- \left \langle x^*,u(s)\right \rangle & \text{ if } s\in [1-t,1],
\\
0  & \text{ if } s \in [0,1-t] .
\end{cases}
\end{align*}
Since 
\[
\int_0^1 f_n(s) ds = 0,
\]
and, the Lesbesgue Dominated Convergence Theorem yields that, for each $0\le t< 1$
\[
0
=
\int^1_{0} f_n(s)ds 
= 
- \int^1_{1-t} \left \langle x^*,u(s) \right \rangle ds.
\]
This yields that 
\[
\left \langle  x^*,u(t) \right \rangle =0
\quad \text{ for each }
t\in [0,1].
\]
Since $x^*\in \mathbb{X}^*$ is an arbitrary functional, this shows that 
\[
u(t)=0 \quad \text{ for each } t\in [0,1].
\]
This shows the uniqueness of the evolution family. And the theorem is proved.
\end{proof}

\section{Roughness of exponential stability and dichotomy under perturbation}
In this section, we study the roughness of exponential dichotomy (and exponential stability) of an autonomous evolution equation (or $\mathrm{C}_0$-semigroup) under small perturbations $B(\cdot)$. 
Recall the following notion of exponential dichotomy.
\begin{definition}
An evolution family $\left (U(t,s)\right )_{t \geq s}$ is said to have an \textit{exponential dichotomy} if there exist projections $P(t): \mathbb{X} \to \mathbb{X}$ and constants $M,\,\alpha > 0$ such that:
\begin{enumerate}
\item The map $t \mapsto P(t)x$ is continuous for every $x \in \mathbb{X}$,
\item $P(t)U(t,s) = U(t,s)P(s)$,
\item $\|U(t,s)x\| \le Me^{-\alpha(t-s)}\|x\|$, 
$x \in \operatorname{Im}P(s)$,
\item $\|U(t,s)y\| \ge M^{-1}e^{\alpha(t-s)}\|y\|$,
$y \in \ker P(s)$,
\item $U(t,s)\big|_{\ker P(s)}$ is an isomorphism onto $\ker P(t)$.
\end{enumerate}

The \textit{evolutionary semigroup} $\left (T_U^h\right )_{h \ge 0}$ on $\mathrm{C}_0(\mathbb{R}, \mathbb{X})$ is defined by
\[
\left [T_U^h v\right ](t) = U(t,t-h)v(t-h).
\]
\end{definition}
It is well known (see, e.g., \cite{latmon, min, rau}) that $\left (U(t,s)\right )_{t \geq s}$ has an exponential dichotomy if and only if $T_U^1$ is hyperbolic, that is
\[
\sigma \left (T_U^1\right ) \cap \{ z\in \mathbb{C} : |z|=1\} =\emptyset .
\]
%%%%%%%%%%%
As a consequence of this we have the following:
\begin{theorem}\label{the exo dic}
Assume that $\left (U(t,s)\right )_{t \geq s}$ has an exponential dichotomy and $\left (V(t,s)\right )_{t \geq s}$ is another evolution family in $\mathbb{X}$ such that $\left \| T^1_U-T^1_V\right \|$ is sufficiently small. Then $\left (V(t,s)\right )_{t \geq s}$ has an exponential dichotomy as well.
\end{theorem}

Applying Theorem \ref{the main} repeatedly, we can prove the following for the existence of an evolution family on the whole real line.
\begin{theorem}\label{the R}
Assume that $B(\cdot )$ satisfies \ref{ItemA1} and \ref{ItemA2} for each interval $[a,b]\subset \mathbb{R}$. Then, Eq.~\eqref{1.2} admits an evolution family on the whole real line.
\end{theorem}
The main result of this section is the following
\begin{theorem}\label{RoughnessEDtheR}
Assume that the assumptions of Theorem~\ref{the R} hold and that the semigroup $\left (T_A(t)\right )_{t \geq 0}$ possesses an exponential dichotomy.
Then, the evolution family generated by Eq.~\eqref{1.2} has an exponential dichotomy as well provided that the perturbation $B(\cdot )$ is uniformly continuous in $\mathcal{GL}_A(\mathbb{X})$, and the bound
\[
\omega_1 =\sup_{t\in \mathbb R} \| B(t)\|_A 
\]
is sufficiently small.
\end{theorem}
\begin{proof}
By the characterization of the exponential dichotomy, it suffices to show that if $\omega_1$ is sufficiently small, then, so is the quantity
\[
\sup_{t\in \mathbb{R}} \left \| U(t,t-1)-T_A(1)\right \| .
\]
By Lemmas \ref{lem key} and \ref{lem per}, for each $n$, using the partition constructed in Theorem \ref{the main}, for each $t\in \mathbb R$, we have
\[
\| U_n(t,t-1) -T_A(1) \| \le  e^{4\omega_1}\omega_1.
\]
Therefore, letting $n\to \infty$ gives
\[
\left \| U(t,t-1) -T_A(1) \right \| \le  e^{4\omega_1}\omega_1.
\]
This, together with Theorem~\ref{the exo dic}, implies that for sufficiently small  
\[
\omega_1 =\sup_{t\in \mathbb R}\| B(t)\|_A
\]
the perturbed evolution equation \eqref{1.2} admits an evolution family that has an exponential dichotomy as well.
\end{proof}

%%%%%%%%%%% Complete %%%%%%%%%

\section{Examples}\label{Sect-Exa}

In this final section, we present several examples that illustrate the applicability of our main results.

\begin{example}
Consider the following semigroup $\left (T(t)\right )_{t\ge 0}$ on $\mathrm{L}^1(\mathbb{R}^+)$ defined as
\begin{equation*}
[T(t)f](s) :=\chi _+ (s-t) f(s-t)=
\begin{cases}
f(s-t), & \text{ if } s\ge t,
\\
0, & \text{ if } 0\le s <t,
\end{cases}
\end{equation*}
where $\chi_{+}(\cdot)$ is the characteristic function of $[0,\infty)$. It can be shown that $\left (T(t)\right )_{t\ge 0}$ is strongly continuous (see, e.g., \cite{rab} and its references). Let us denote by $G$ the generator of $(T(t))_{t\ge 0}$. Consider the evolution equation
\[
u'(t)=Gu(t),\quad t\ge 0,
\]
where $G=-\frac{d}{dt}$, $u(t)\in \mathrm{L}^1(\mathbb{R}^+)$. As is well known, the "general solution" of this equation is the translation semigroup $\left (T(t)\right )_{t\ge 0}$ in $\mathrm{L}^1(\mathbb{R})$. This is a $\mathrm{C}_0$-semigroup with the generator $G:=- \frac{d}{dt}$, with domain $D(G)$ defined as
\[
D(G) =\{ v(\cdot ) : v(\cdot )\in \mathrm{L}^1(\mathbb{R}^+) \text{ absolutely continuous and } v'(\cdot) \in \mathrm{L}^1(\mathbb{R}^+)\} =\mathrm{W}^{1,1}(\mathbb{R}^+).
\]

We will find functions $b\in \mathrm{L}^{1}(\mathbb{R}^+)\setminus \mathrm{L}^2(\mathbb R^+)$. This means, the operator $B$ in $\mathrm{L}^1(\mathbb{R}^+)$ defined as 
\[
[Bx](t)= b(t)x(t), \qquad t\in \mathbb{R},\quad x\in \mathrm{L}^1(\mathbb{R}^+),
\]
may not be a bounded operator. 
Also, $B$ will be found to be closed.
Since $B$ is closed (see \cite[Example]{miy}, the integral
\[
\int^\infty _0 e^{-\mu s} b(\xi )\chi_+(s-\xi)x(s-\xi)ds
\]
is convergent and 
\[
R(\mu,G) x(\xi) = \int^\infty _0 e^{-\mu s} \chi_+(s-\xi)x(s-\xi)ds
\]
we have
\begin{align*}
BR(\mu,G) x(\xi)  &=\int^\infty_0 e^{-\mu s} [BT(s)x](\xi)ds
\\	
&=\int^\infty _0 e^{-\mu s} b(\xi )\chi_+(s-\xi)x(s-\xi)d s.
\end{align*}
Therefore,
\begin{align*}
\|	BR(\mu,G) x(\xi) \|_1 	&=\int^\infty_0 \left|\int^\infty _0 e^{-\mu s} b(\xi )\chi_+(s-\xi) x(s-\xi)d s\right| d\xi 
\\
&\le  \int^\infty _0 e^{-\mu s}   \int^\infty_{0}  \left| b(\xi )\chi_+(s-\xi)x(s-\xi)d \xi\right| ds
\\
&=  \int^\infty _0 e^{-\mu s}   \int^s_{0}  \left| b(\xi )x(s-\xi)d \xi\right| ds
\\
&\le   \int^\infty _0 e^{-\mu s} \left( |b|*|x|\right) (s) ds
\\
&\le |B(\mu)| |X(\mu)|,
\end{align*}
where $B(\mu)$ and $X(\mu)$ are the Laplace transforms of $b$ and $|x|$, respectively.
As is well known, 
\[
|X(\mu)| \le \| x\|_1
\quad \text{ for } \mu >0.
\]
Therefore, for 
\begin{equation}\label{5-7}
\|	BR(\mu,G) \|  \le \frac{K}{\mu-\omega_0}, \quad \mu >\omega ,
\end{equation}
where $K$ and $\omega_0$ are certain fixed constants, independent of $\mu$, we just need to choose $b$ such that  
\begin{equation}\label{5-10}
\limsup_{\mu \to \infty} \left |\mu^\gamma B(\mu)\right | <\infty ,
\end{equation}
where $\gamma \ge 1$ is any fixed number.

According to \cite[Theorem 1, p.~181]{wid}, with the normalized 
\[
\alpha (t) :=\int^t_0 b(s)ds,
\]
for \eqref{5-10} to be valid we can choose $b$ such that
\[
\limsup_{t\to 0^+} \left| \frac{\alpha (t)}{t}  \right| =	
\limsup_{t\to 0^+} \left| \frac{1}{t} \int^t_0 b(s)ds \right| 
<
\infty.
\]
A simple candidate for such a function $b$ is the following
\begin{equation}\label{5-9}
b(t):=\sum_{n=1}^\infty n^2 \chi_{[n,n+n^{-4}]}(t),
\end{equation}
where $\chi_{[n,n +n^{-4}]}(\cdot)$ is the characteristic function of the interval $[n,n+n^{-4}]$. 
\begin{claim}
For the function $b$ defined in \eqref{5-9}, the following assertions are true:
\begin{enumerate}
\item $b\in \mathrm{L}^1([0,\infty))$, $D(B) \supset D(G)$;
\item \eqref{5-10} is valid;
\item $b$ is an unbounded operator;
\item $b$ is a closed operator.
\end{enumerate}
\end{claim}
\begin{proof}
\begin{enumerate}
\item The first assertion is true because of the identity
\[
\int^\infty_0 |b(t)|dt = \sum_{n=1}^\infty \frac{1}{n^2} < \infty .
\]
Further, if $f\in D(G)=\mathrm{W}^{1,1}(\mathbb{R}^+)$, 
\[
f(t)=f(0)+\int^t_0f'(s)ds,
\]
so 
\[
|f(t)| \le \| f'\|_1
\quad \text{ a.e. } t\in \mathbb{R}^+,
\]
or $f\in \mathrm{L}^\infty([0,\infty))$. Therefore, 
\[
b(\cdot )f(\cdot) \in \mathrm{L}^1([0,\infty)).
\]
This implies that $D(G)\subset D(B)$.

\item Since $b(t)=0$, for $ t\in [0,1]$, by \cite[Theorem 1, p. 181]{wid}, with the normalized 
\[
\alpha (t) :=\int^t_0 b(s)ds,
\]
we have
\[
\limsup_{\mu\to\infty} |\mu^\gamma  B(\mu)|
\le 	\limsup_{t\to 0^+} \left| \frac{\alpha (t  )}{t^\gamma }  \right|
 =	 \limsup_{t\to 0^+} \left| \frac{1}{t^\gamma } \int^t_0 b(s)ds \right| 
 =0.
\]

\item It suffices to show that $\mathrm{L}^1([0,\infty))\setminus  D(B)\not= \emptyset$. In fact, as 
\[
b(t)\cdot b(t) =  \sum_{n=1}^\infty n^4 \chi_{[n,n+n^{-4}]}(t),
\]
we have 
\[
\int^\infty_0 |b^2(t)|dt 
\ge \sum_{n=1}^\infty {n^4\cdot n^{-4}} 
=\infty .
\]
In other words, $b\in\mathrm{L}^1([0,\infty))\setminus  D(B)$. This implies $\mathrm{L}^1([0,\infty))\setminus  D(B)\not= \emptyset$.

\item The closedness of $B$ can be shown following the lines of \cite[Example, p. 309]{miy}.
\end{enumerate}
The proof is complete.
\end{proof}

\begin{corollary}
The evolution equation
\begin{equation*}
u'(t)=(G+B)u(t), \quad u(t)\in \mathbb{X} ,
\end{equation*}
where $\mathbb{X} :=\mathrm{L}^1([0,\infty))$, is well-posed. In other words, the perturbed operator $G+B$ generates a $\mathrm{C}_0$-semigroup in $\mathrm{L}^1([0,\infty))$.
\end{corollary}
\begin{proof}
The operator $B$ satisfies \eqref{5-7}, so \cite[Theorem 2.4, p. 493]{buihuyluomin} applies.
\end{proof}

We next define
\[
B(t) := \sin (t)B, \quad t\in [0,2\pi].
\]
Note that for each fixed $t\in [0,2\pi]$,
$B(t)$ is a closed operator (see \cite[Example]{miy}). Since $B(t)$ is a closed operator for each $t$ 
\begin{align*}
B(t)R(\mu,G) x(\xi)  &=\int^\infty _0 e^{-\mu s} [B(t)T(s)x](\xi)ds
\\	
&=\sin (t)  \int^\infty _0 e^{-\mu s} b(\xi )x(s+\xi)ds .
\end{align*}
Therefore, for $\mu>\omega_0$
\begin{align}
\| B(t)R(\mu,G)x(\cdot ) \|_1 &=\sin (t) \int^\infty_{-\infty}\left|  \int^\infty _0 e^{-\mu s}b(\xi) x(s+\xi)ds \right|d\xi \nonumber\\
&= \sin (t)  \int^\infty _0  e^{-\mu s} \left| \int^\infty_{-\infty} b(\xi) x(s+\xi) d\xi \right|  ds \notag\\
&\le  \frac{K\| x\|_1 }{\mu-\omega_0}. \label{5-16}
\end{align}
\begin{corollary}
The nonautonomous evolution equation
\[
u'(t) = (G+B(t))u(t), \quad u(t)\in \mathbb{X},
\]
where $\mathbb{X} :=\mathrm{L}^1([0,\infty))$ is well-posed. That is, this equation admits a unique evolution family $\left (U(t,s)\right )_{0\le s\le t \le 1}$.
\end{corollary}
\begin{proof}
As shown in \eqref{5-16}, $B(\pi/2)R(\mu,G) \in \mathcal{GL}_G(\mathbb{X})$.
Moreover,
\[
B(t)R(\mu,G) =\sin(t)B(\pi/2)R(\mu,G)
\]
is continuous in $t\in [0,2\pi]$. Therefore,
\ref{ItemA1} is satisfied for $B(\cdot )$.

Since $B(t)=\sin(t)B$, for each $t\in [0,2\pi]$, $B(t)R(\mu,G)\in \mathcal{L}(\mathbb{X})$
\[
B'(t)R(\mu,G)= \cos (t) BR(\mu,G).
\]
By \eqref{5-16},
\begin{align*}
\|B'(t)R(\mu,G)\| 
& \le 
\|BR(\mu,G)\| 
\\
& \le 
\frac{K}{\mu-\omega_0}.
\end{align*}
This yields that \ref{ItemA2} is satisfied. Therefore, Theorem~\ref{the main} applies.
\end{proof}
\end{example}

\begin{example}
Consider the following problem 
\begin{equation}\label{exa 3}
\left\{
\begin{aligned}
\frac{\partial u(x,t)}{\partial t} 
&= 
\frac{\partial^2 u(x,t)}{\partial x^2} +\sin (t) b(x)u(x,t),
&&
x\in \mathbb{R},
&& t \ge s,
\\
u(x,s) &=\varphi (x),
&&
x\in \mathbb{R}.
\end{aligned}
\right .
\end{equation}
Here, $s\in \mathbb{R}$ and $b(\cdot )\in \mathrm{L}^1(\mathbb{R})$ such that 
\[
b(t):= b_+(t)+b_-(t),
\]
where
\[
b_+(t):=  
\sum_{n=1}^\infty n^2 \chi_{[n,n+n^{-4}]}(t),\quad
b_-(t):= b_+(-t).
\]
\end{example}
We define $\mathbb{X}:=\mathrm{L}^1(\mathbb{R})$, $A:= \frac{d^2}{dx^2}$ where
\[
D(A):=\{ y\in \mathbb{X} \colon y',\, y'' \text{ exist as elements in } \mathbb{X}\}.
\]
%so $D(A)$ consists of all functions $y\in \mathbb{X}$ that are absolutely continuous together with their derivative $y'$ such that  $y'$ and $y''$ are elements of $\mathbb{X}$,
%and 
It is well known that $D(A)$ is dense everywhere in $\mathrm{L}^1(\mathbb{R})$. For each $t \in \mathbb{R}$ we define an operator of multiplication $B$ in $\mathrm{L}^1(\mathbb{R})$ by
\[
[By](x)=b(x)y(x)
\]
and 
\[
B(t):=\sin (t)B
\]
for each $t \in \mathbb{R}$ and $y\in \mathrm{L}^1(\mathbb{R})$ such that $b(\cdot )y(\cdot ) \in \mathrm{L}^1(\mathbb{R})$. 
Generally, $B(t)$ is not an everywhere defined operator, so is not a bounded operator. 
In fact, when $\sin (t)\not=0$, $B(t)$ is not defined at $b\in \mathrm{L}^1(\mathbb R)$ because $b\not\in \mathrm{L}^2(\mathbb R)$.
\begin{claim}
The operator $B(t)$ is a linear operator in $\mathbb{X}$ and 
$D(A) \subset D(B(t))$.	
\end{claim}	
\begin{proof}
Obviously, $B(t)$ is a linear operator in $\mathbb{X}$. Next, let $\phi\in D(A)$. Clearly, 
$\phi \in \mathrm{L}^\infty (\mathbb{R})$ because
\begin{align*}
| \phi(t)| 
& \le \int^t_0 |\phi'(\xi)|d\xi +|\phi(0)| 
\\
& \le \| \phi'(\cdot )\|_1 +|\phi(0)| ,
\end{align*}
so 
\[
\sin(t) b(\cdot )\phi (\cdot)\in \mathrm{L}^1(\mathbb{R}).
\]
This yields that $D(A)\subset D(B(t))$.	
\end{proof}
Next, we note that the following claim is well known in the theory of semigroups (see \cite[Corollary II4.9, p.~106]{engnag}). However, 
by a direct calculation of $R(\mu,A)$ it can easily be shown and this will be subsequently used to verify Theorem \ref{the main}'s conditions. 
\begin{claim}
Linear operator $A$ generates a $\mathrm{C}_0$-semigroup in $\mathrm{L}^1(\mathbb{R})$.
\end{claim}
\begin{proof}
We will use elementary facts about the existence and uniqueness of a bounded solution of inhomogeneous linear equations on the line (see for example, \cite{mal}). Consider the equation
\[
\mu y -y'' = f.
\]
Using the Green's function method, one obtains that for each $\mu>0$ and $f\in \mathrm{L}^1(\mathbb{R})$
\[
R(\mu, A)f (t)
= \frac{1}{2\sqrt{\mu}}
\left(
\int^t_{-\infty} e^{-\sqrt{\mu}(t-s)}f(s) ds
+
\int^{\infty}_t e^{\sqrt{\mu}(t-s)}f(s) ds
\right).
\]
Before proceeding we denote by
\[
h_1(s) := \begin{cases}
e^{-\sqrt{\mu} s} & \text{ if } s \le 0,\\
0 & \text{ if } s> 0,
\end{cases}
\]
and
\[
h_2(s) := \begin{cases}
0 & \text{ if } s \le 0,
\\
e^{\sqrt{\mu}s} & \text{ if } s>0.
\end{cases}
\]
Then, $h_1,\,h_2\in \mathrm{L}^1(\mathbb{R})$ with 
\[
\| h_1\|_1 = \frac{1}{\sqrt{\mu}}= \| h_2\|_1.
\]
Therefore, by Young's inequality
\begin{align*}
\| R(\mu, A)f \|_1 & = \frac{1}{2\sqrt{\mu}}\int^\infty_{-\infty} \left| \int^t_{-\infty} e^{-\sqrt{\mu} (t-s)}|f(s)|ds +\int^\infty_t e^{\sqrt{\mu}(t-s)} |f(s)|ds    \right|  dt 
\\
&=\frac{1}{2\sqrt{\mu}} \int^\infty_{-\infty}  \left( \int^\infty_0 e^{-\sqrt{\mu} \eta} |f(t-\eta)|d\eta  +\int^0_{-\infty} e^{\sqrt{\mu}\eta} |f(t-\eta)| d\eta       \right) dt
\\
&\le \frac{1}{2\sqrt{\mu}}  \int^\infty_{-\infty}  \left( \int^\infty_{-\infty}| h_1(\eta ) f(t-\eta)|d\eta  +\int^\infty_{-\infty} |h_2(\eta) f(t-\eta) |d\eta       \right) dt 
\\
&\le  \frac{1}{2\sqrt{\mu}}  (\|h_1\|_1+\|h_2\|_1) \| f\|_1 
\\
&= \frac{\|f\|_1}{\mu}.
\end{align*}
Since $D(A)$ is dense everywhere in $\mathrm{L}^1(\mathbb{R})$, this estimates, by the Hille--Yosida Theorem, shows that $A$ generates a $\mathrm{C}_0$-semigroup in $\mathrm{L}^1(\mathbb{R})$. 
\end{proof}
We next estimate $\| BR(\mu,A)\|$ to show that $B(\cdot )$ satisfies \ref{ItemA1} of Theorem \ref{the main}. 
In fact, we have
\begin{align*}
& \| BR(\mu, A)f \|_1
\\
& =
\frac{1}{2\sqrt{\mu}} \int^\infty_{-\infty}| b(t)| \left|
\int^t_{-\infty} e^{-\sqrt{\mu}(t-s)}f(s) ds
+
\int^{\infty}_t e^{\sqrt{\mu}(t-s)}f(s) ds
\right|dt
\nonumber\\
&\le \frac{1}{2\sqrt{\mu}} \int^\infty_{0} (|b_+(t)|+|b_-(t)| )   \left|
\int^\infty_{0} e^{-\sqrt{\mu}\eta}f(t-\eta) d\eta
-
\int^{\infty}_0 e^{-\sqrt{\mu}\eta }f(t+\eta)d\eta  
\right|dt .
\end{align*}
We have 
\begin{align*} 	
 \int^\infty_{0} |b_+(t)|    
\int^\infty_{0} e^{-\sqrt{\mu}\eta}  \left|f(t-\eta) \right|  d\eta 
dt &= 
 \int^\infty_{0}  e^{-\sqrt{\mu}\eta} 
\int^\infty_{0}   |b_+(t)|   \left| f(t-\eta) \right|  dt d\eta 	.
\end{align*}
Set $u=t-\eta$,
\begin{align*} 	 
 \int^\infty_{0}  e^{-\sqrt{\mu}\eta } 
\int^\infty_{0}   |b_+(t)|   \left| f(t-\eta) \right|  dt d\eta 	 
 &=  \int^\infty_{0}  e^{-\sqrt{\mu}(t-u)} 
\int_{-\infty}^{0}   |b_+(t)|   \left| f(u) \right| du	 dt \\
&=  \int^\infty_{0}  e^{-\sqrt{\mu}t}  |b_+(t)|  dt
\int_{-\infty}^{0}  e^{\sqrt{\mu}u}   \left| f(u) \right|  dt du \\
&\le B_+(\sqrt{\mu}) \| f\|_1 ,
\end{align*}
where $B_+(s)$ is the Laplace transform of $b_+(\cdot)$. Similarly,
 \begin{align*}
\int^\infty_{0} |b_+(t)|     \left|
\int^{\infty}_0 e^{-\sqrt{\mu}\eta }f(t+\eta)d\eta   \right| dt 
 \le B_+(\sqrt{\mu}) \| f\|_1.
\end{align*}
Therefore,
\begin{align} \label{5-21}
\int^\infty_{0} |b_+(t)|     \left|
\int^t_{-\infty} e^{-\sqrt{\mu}(t-s)}f(s) ds
\right|dt \le 2B_+(\sqrt{\mu}) \| f\|_1.
\end{align}
Similarly, we can show that
\begin{align} \label{5-22}
\int^\infty_{0} |b_-(t)|     \left|
\int^t_{-\infty} e^{-\sqrt{\mu}(t-s)}f(s) ds
\right|dt \le 2B_+(\sqrt{\mu}) \| f\|_1.
\end{align}
As we have chosen $b$ such that $b_+$ is exactly the function $b$ in the previous example, by \eqref{5-10} there exist constants $K,\, \omega_1>0$ such that
\[
| B_+(s) | \le \frac{K}{s^2 -\omega_1}.
\]
Consequently, by \eqref{5-21} and \eqref{5-22} we see that
\begin{align*}
\| BR(\mu, A)f \|_1 
&\le 
 4B_+(\sqrt{\mu}) \| f\|_1\\
&\le  
 \frac{4K}{\mu  -\omega_1}\| f\|_1,
\end{align*}
so, there exist constants $K$ and $\omega_1$ such that for all $\mu>\omega_1$
\begin{equation}\label{5-26}
\| BR(\mu, A)f \| \le \frac{4K}{\mu-\omega_1}.
\end{equation}
\begin{corollary}
With above notations and assumptions on $b$, the operator $B$ is unbounded and the perturbation operator $A+B$ generates a $\mathrm{C}_0$-semigroup in $\mathbb{X}$.
\end{corollary}
\begin{proof}
As $D(B)\supset D(A)$ and \eqref{5-26} is valid, the condition of Theorem 2.4 in \cite[Theorem 2.4]{buihuyluomin} is satisfied. This implies that the perturbation operator $A+B$ generates a $\mathrm{C}_0$-semigroup.
\end{proof}
The above-mentioned operator $B$ is an example of an unbounded operator in $\mathbb{X}$ that is in $\mathcal{GL}_A(\mathbb{X})$.

Next, we will consider the nonautonomous perturbation $A+B(t)$.
\begin{claim}
The operator $B(t)$, $t\in \mathbb{R}$, satisfies \ref{ItemA1} and \ref{ItemA2} for each interval $[a,b]\subset \mathbb R$, so Eq.~\eqref{exa 3} has a unique solution $u(x,t)$, $t\ge s$, for each $\varphi (\cdot)$ in $\mathbb{X}$.
\end{claim}	
\begin{proof}
By \eqref{5-26}, it is easy to see that $B(\cdot)$ satisfies \ref{ItemA1} on each $[a,b]\subset \mathbb R$ because 
\[
B(t)=\sin(t)B
\]
and the continuity of $B(\cdot)$ follows from the fact that $\sin(t)$ is continuous in $t\in \mathbb{R}$. Also, note that $B(t)R(\mu,A)$ is continuously differentiable in $t$, and as
\[
B'(t)=\cos (t)B
\]
and by \eqref{5-26} we see that
\begin{align*}
\limsup_{\mu\to \infty} \sup_{t\in\mathbb{R}} \left \| \frac{d}{dt}\left[B(t)R(\mu,A)\right]\right \| &\le 
\limsup_{\mu\to \infty} \sup_{t\in\mathbb{R}} \| BR(\mu,A)\| \\
&\le 	\limsup_{\mu\to \infty}   \frac{4K}{\mu-\omega_1}
\\
& = 0.
\end{align*}
Finally, Theorem \ref{the main} applies.
\end{proof}

\subsection*{Acknowledgements}
The first-named author (X.-Q. Bui) thanks the Vietnam Institute for Advanced Study in Mathematics (VIASM) for financial support and a stimulating working environment.

\subsection*{Data availability statement}
No data was used for the research described in the manuscript.

\subsection*{Declarations}
\paragraph*{Conflict of interest}
The authors declare that there is no conflict of interest.

%\bibliographystyle{amsplain}
 
% =============================
% =============================
\end{document}